\documentclass[11pt]{article}

\usepackage[margin=1.3in]{geometry}
\usepackage{amssymb,amsmath}
\usepackage{amsfonts}
\usepackage{amsthm}
\usepackage{hyperref, xcolor}
\usepackage{mathrsfs}
\usepackage{textcomp}

\usepackage{verbatim}

\usepackage{sectsty}

\hypersetup{
    colorlinks,%
    citecolor=black,%
    filecolor=black,%
    linkcolor=black,%
    urlcolor=black
}

\makeatletter

\newdimen\bibspace
\renewenvironment{thebibliography}[1]{%
 \section*{\refname 
       \@mkboth{\MakeUppercase\refname}{\MakeUppercase\refname}}%
     \list{\@biblabel{\@arabic\c@enumiv}}%
          {\settowidth\labelwidth{\@biblabel{#1}}%
           \leftmargin\labelwidth
           \advance\leftmargin\labelsep
           \itemsep\bibspace
           \parsep\z@skip     %
           \@openbib@code
           \usecounter{enumiv}%
           \let\p@enumiv\@empty
           \renewcommand\theenumiv{\@arabic\c@enumiv}}%
     \sloppy\clubpenalty4000\widowpenalty4000%
     \sfcode`\.\@m}
    {\def\@noitemerr
      {\@latex@warning{Empty `thebibliography' environment}}%
     \endlist}

\makeatother

\makeatletter

\newtheorem{thm}{Theorem}[section]
\newtheorem{lem}[thm]{Lemma}
\newtheorem{prop}[thm]{Proposition}

\newtheorem{cor}[thm]{Corollary}
\newtheorem{rem}[thm]{Remark}

\def\Xint#1{\mathchoice
  {\XXint\displaystyle\textstyle{#1}}%
  {\XXint\textstyle\scriptstyle{#1}}%
  {\XXint\scriptstyle\scriptscriptstyle{#1}}%
  {\XXint\scriptscriptstyle\scriptscriptstyle{#1}}%
  \!\int}
\def\XXint#1#2#3{{\setbox0=\hbox{$#1{#2#3}{\int}$}
  \vcenter{\hbox{$#2#3$}}\kern-.5\wd0}}

\def\dashint{\Xint-}

             \newcommand{\lda}{\lambda}
                \newcommand{\pa}{\partial}
\newcommand{\va}{\varepsilon}           \newcommand{\ud}{\mathrm{d}}
\newcommand{\be}{\begin{equation}}      \newcommand{\ee}{\end{equation}}
                 \newcommand{\X}{\overline{X}}
\newcommand{\Lda}{\Lambda}              \newcommand{\B}{\mathcal{B}}
\newcommand{\R}{\mathbb{R}}

\newcommand{\MH}{\mathcal H}

\newcommand{\al}{\alpha}
\newcommand{\om}{\Omega}

\newcommand{\calC}{{\mathcal C}}

\newcommand{\calO}{{\mathcal O}}

\DeclareMathOperator\capac{Cap}

\DeclareMathOperator\dist{dist}

\begin{document}

\title{\textbf{On local behavior of singular positive solutions to nonlocal elliptic equations}
\bigskip}

\author{\medskip   Tianling Jin\footnote{Supported in part by Hong Kong RGC grant ECS 26300716.}, \  \  Olivaine S. de Queiroz\footnote{Partially supported by CNPq-Brazil.},\ \ Yannick Sire, \ \ Jingang Xiong\footnote{Supported in part by NSFC 11501034, NSFC 11571019,  Beijing MCESEDD (20131002701) and the Fundamental Research Funds for the Central Universities.}}

\date{\today}

\maketitle

\begin{abstract}   We study local behavior of positive solutions to the fractional Yamabe equation with a  singular set of fractional capacity zero. 

\end{abstract}


\section{Introduction}

In the classical paper \cite{CGS}, Caffarelli-Gidas-Spruck studied the local behavior of positive solutions of
\be \label{eq:cl0}
-\Delta u= g(u) \geq 0
\ee
in the punctured unit ball $B_1\setminus \{0\}$ of $\R^n$, $n\ge 3$.
With some condition on the nonlinear function $g(t)$, they proved that every local solution $u$ is asymptotically radially symmetric, and showed that $u$ has a precise behavior near the isolated singularity $0$. Typical examples of $g$ are
$
g(t)=t^{p}, \ \ \frac{n}{n-2}\leq p\leq \frac{n+2}{n-2}.
$
Such equations are of Yamabe type equations with isolated singularities, and they have attracted a lot of attention. We refer the reader to \cite{CGS,HLT, KMPS,Lic,Li06,MP,Z} and references therein. In \cite{ChLin}, Chen-Lin studied a more general case that is the equation \eqref{eq:cl0} in $B_1\setminus\Lda$, where $\Lda$ is a singular set other than a single point. The importance of studying solutions of \eqref{eq:cl0} with a singular set was indicated in the work of Schoen \cite{Schoen} and Schoen-Yau \cite{SY} on complete locally conformally flat manifolds. 

In this paper, we are interested in the positive singular solutions of the fractional Yamabe equation
\be\label{eq:maineq}
(-\Delta)^\sigma u=u^{\frac{n+2\sigma}{n-2\sigma}} \quad \mbox{in }\om \setminus \Lda, \quad u>0\quad\mbox{in }\R^n,
\ee
where $n\ge2$, $\sigma\in (0,1)$, $\om$ is an open set in $\R^n$, $\Lda$ is a closed subset of measure zero, and $(-\Delta)^\sigma$ is the fractional Laplacian defined as
\begin{equation}\label{eq:fl}
(-\Delta )^\sigma u(x)= \mbox{P.V.} c_{n,\sigma} \int_{\R^n}  \frac{u(x)-u(y)}{|x-y|^{n+2\sigma}}\,\ud y
\end{equation}
with $c_{n,\sigma}=\frac{2^{2\sigma}\sigma\Gamma(\frac{n+2\sigma}{2})}{\pi^{\frac{n}{2}} \Gamma(1-\sigma)}$ and the gamma function $\Gamma$. Throughout this paper, we assume that 
$$
u\in C^2(\om \setminus \Lda) \ \ \textrm{and} \ \ \int_{\R^n}\frac{|u(x)|}{1+|x|^{n+2\sigma}}\,\ud x<\infty,
$$
which will make the formula \eqref{eq:fl} well-defined in $\Omega\setminus\Lambda$. Each solution $u$ of \eqref{eq:maineq} induces a conformal metric $g:=u^{\frac{4}{n-2\sigma}} |\ud x|^2$ of constant fractional $Q$-curvature \cite{CG} in $\om \setminus \Lda$.  In view of the singular Yamabe problem, one may ask that if $\Lda \subset \om$ is a $k$-dimensional smooth compact manifold, can we construct a complete conformal metric $g$ of constant fractional $Q$-curvature? Can we describe asymptotic behavior of the singular (not necessary complete) conformal metrics of constant fractional $Q$-curvature? Due to the nonlocality, they are hard to answer. Under some conditions, Gonz\'alez-Mazzeo-Sire \cite{GMS} showed that $\Gamma(\frac{n}{4} - \frac{k}{2} + \frac{\sigma}{2}) \Big/ \Gamma(\frac{n}{4} - \frac{k}{2} - \frac{\sigma}{2}) > 0$ is necessary to have a complete metric (see Theorem \ref{th:SY}). They also constructed complete metrics when $\sigma$ is very close to $1$ and established a blow up rate. When $\Lda$ is an isolated point,  Caffarelli-Jin-Sire-Xiong \cite{paper-ARMA} proved asymptotic radial symmetry of the singular solutions and their sharp blow up rate . The radial singular solutions have been studied by  DelaTorre-Gonz\'alez \cite{DG} and DelaTorre-del Pino-Gonzalez-Wei \cite{DD+}. In particular, in \cite{DD+} they constructed a class of Delaunay-type solutions. There are other work on the singular Yamabe problem, for example, Qing-Raske \cite{qing-raske} and Zhang \cite{zhang}. For the fractional Yamabe problem on compact manifolds, we would like to refer to the work Gonz\'alez-Qing \cite{gonzalez-qing}, Choi-Kim \cite{choi-kim}, Gonz\'alez-Wang \cite{gonzalez-wang} and Kim-Musso-Wei \cite{kim-musso-wei-1, kim-musso-wei}.



To analyze \eqref{eq:maineq}, we will use the fact that the fractional Laplacian $(-\Delta)^\sigma$ can also be realized as a Dirichlet-to-Neumann operator. This was discovered by Caffarelli-Silvestre \cite{CaS}.  In order to describe in a more precise way, let us first introduce some notations. We use capital letters, such as $X=(x,t)$, to denote points in $\R^{n+1}$, and $t\geq 0$ usually.
$\B_R(X)$ denotes the ball in $\R^{n+1}$ with radius $R$ and center $X$, $\B^+_R(X)$ as $\B_R(X)\cap \R^{n+1}_+$, and $B_R(x)$ as the ball in $\R^{n}$ with radius $R$ and center $x$. We also write $\B_R(0), \B^+_R(0), B_R(0)$ as $\B_R, \B_R^+, B_R$ for short. We use $\pa' \B_R^+(X), \pa'' \B_{R}^+(X)$ to denote the straight and curved boundary portion of $\pa \B_{R}^+(X)$, respectively. Through the extension formulation for $(-\Delta)^\sigma$ in \cite{CaS}, the equation \eqref{eq:maineq} is equivalent to a degenerate elliptic equation with a Neumann boundary condition in one dimension higher:
\be\label{eq:ex0-1}
\begin{cases}
\mathrm{div}(t^{1-2\sigma} \nabla_{X} U)=0 & \quad \mbox{in }\R^{n+1}_+,\\
\frac{\pa U}{\pa \nu^\sigma} = u^{\frac{n+2\sigma}{n-2\sigma}} &\quad \mbox{for }x\in \om\setminus\Lambda,
\end{cases}
\ee
where
$$
\frac{\pa U}{\pa \nu^\sigma}(x,0)= -\lim_{t\to 0^+} t^{1-2\sigma} \pa_t U(x,t),
$$
and $u(x)=U(x,0)$. 

For an open set $E\subset\R^{n+1}$, we define the weighted Sobolev space $W^{1,2}(|t|^{1-2\sigma}, E)$ as the space of weakly differentiable $L^1$ functions with bounded norm
\[
\|U\|_{W^{1,2}(|t|^{1-2\sigma},E)}:= \left(\int_{E} |t|^{1-2\sigma} U^2\,\ud X+\int_{E} |t|^{1-2\sigma} |\nabla_X U|^2\,\ud X \right)^{1/2}.
\]
The weight $|t|^{1-2\sigma}$ belongs to the $A_2$ class and the weighted Sobolev space is well understood; see Fabes-Jerison-Kenig \cite{fabes} and  the book Heinonen-Kilpel\"ainen-Martio \cite{HKM}. A solution of \eqref{eq:ex0-1} is understood as a function in $W^{1,2}(|t|^{1-2\sigma},K)$ for every compact set $K\subset\R^{n+1}_+\cup\{\Omega\setminus\Lda\}$ satisfying \eqref{eq:ex0-1} in the sense of distribution. Many regularity properties for such weak solutions of linear equation related to \eqref{eq:ex0-1} can be found in Cabre-Sire \cite{CS}, Jin-Li-Xiong \cite{JLX} and etc.

Our first theorem is a cylindrical  symmetry result when $\om$ is the whole space and $\Lda$ is a lower dimensional hyperplane. Namely,

\begin{thm}\label{symTemp}
Let $1\le k\le n-2\sigma$ and $U$ be a nonnegative solution of
\begin{equation}\label{temp}
\begin{cases}
\mathrm{div}(t^{1-2\sigma} \nabla_{X} U)=0 & \quad \mbox{in }\R^{n+1}_+,\\
\frac{\pa U}{\pa \nu^\sigma} = U(x,0)^{\frac{n+2\sigma}{n-2\sigma}} &\quad \mbox{on }\R^n\setminus\R^k.
\end{cases}
\end{equation}
Suppose there exists $x_0\in\R^{k}$ such that $\limsup_{\xi\to(x_0,0)}U(\xi)=\infty$. Then
$$
U(x',x'',t)=U(x',\tilde x'',t)
$$
where $x' \in \R^k$ and $x'',\tilde x''\in \R^{n-k}$ that $|x''|=|\tilde x''|$.
\end{thm}

The condition $k\le n-2\sigma$ will ensure that $\capac_\sigma(\Lambda)=0$ (see \eqref{def:capacity} and Theorem \ref{thm:hausdroff1}). If there is no singular point of $U$ in $\R^{n+1}_+\cup \pa \R^{n+1}_+$, Jin-Li-Xiong \cite{JLX} proved a Liouville theorem. If there is only one singular point on $\pa \R^{n+1}_+$, Caffarelli-Jin-Sire-Xiong \cite{paper-ARMA} prove that $U(x,0)$ is radially symmetric.  If $\om$ is not $\R^n$, one should not expect to have the cylindrical symmetry. However, we can show an asymptotic cylindrical symmetry. In fact, we can prove it when $\Lda$ is a smooth submanifold of $\R^n$.  To this end, we assume that $\Lambda\subset B_{1/2}$ is a smooth $k-$dimensional closed manifold with
$
k \leq n-2\sigma.
$
Let $N$ be a tubular neighborhood of $\Lambda$ such that any point of $N$ can be uniquely expressed as the sum $x+v$ where $x \in \Lambda$ and $v \in (T_x\Lambda)^\perp$, the orthogonal complement of the tangent space of $\Lambda $ at $x$. Denote $\Pi$ the orthogonal projection of $N$ onto $\Lambda$. For small $r>0$ and $z \in \Lambda$,
$$
\Pi_r^{-1} (z)= \left \{ y \in N, \,\,|\,\,\Pi(y)=z,\ |y-z|=r\right \}.
$$
We prove the following

\begin{thm}\label{thm:sym}
Suppose $U\geq 0$ in $\B_2^+$ is a solution of
\be\label{eq:ex0}
\begin{cases}
\mathrm{div}(t^{1-2\sigma} \nabla_{X} U)=0 & \quad \mbox{in }\B_2^+,\\
\frac{\pa U}{\pa \nu^\sigma} = U^{\frac{n+2\sigma}{n-2\sigma}} &\quad \mbox{for }x\in \pa'\B_2^+\setminus\Lambda,
\end{cases}
\ee
and $N$, $\Lambda$ and $\Pi$ are as above. Then we have, for $x,x' \in \Pi_r^{-1}(z)$,
\begin{equation}\label{sym}
U(x,0)=U(x',0)(1+O(r))\quad\mbox{as } r \to 0^+,
\end{equation}
where $O(r)$ is uniform for all $z \in \Lda$.
\end{thm}

Since we do not use any special structure of the half ball, $\B_2^+$ can be replaced by general open sets containing $\B_{1/2}^+$.
When $\Lda$ is a point, the above theorem has been proved in Caffarelli-Jin-Sire-Xiong \cite{paper-ARMA}.

Finally, we provide an asymptotic blow up rate estimate for solutions with a singular  set of fractional capacity zero, which is not necessary to be a smooth manifold. Let us introduce the fractional capacity. For every compact subset $\Lda$  of $\R^n$ and $0<\sigma<1$, define
\begin{equation}\label{def:capacity}
\text{Cap}_\sigma(\Lda):=\inf \left \{ \int_{\mathbb R^n} |\xi|^{2\sigma} |\hat f(\xi)|^2\,\ud \xi:\,\,f \in C^\infty_c(\mathbb R^n),\,\,f(x)\geq 1  \text{ in }\Lda \right \}.
\end{equation}
This is a modification of the classical Newtonian capacity for our purpose. By the Caffarelli-Silvestre's extension formula, we will give an equivalent definition in Section \ref{sec:pre}. Two properties on the relation between this capacity $\text{Cap}_\sigma$ and the Hausdorff dimension are presented in Theorems \ref{thm:hausdroff1} and \ref{thm:hausdroff2}.

\begin{thm}\label{thm:a} Let $\Lda\subset B_{1/2}$ be compact and $\capac_\sigma(\Lambda)=0$. Let $U$ be a nonnegative solution of \eqref{eq:ex0}. Then there exists $C>0$ such that
\be \label{eq:clc}
u(x)\le C\dist (x,\Lda)^{-\frac{n-2\sigma}{2}}
\ee
for all $x\in B_2\setminus\Lda$.
\end{thm}

\begin{rem}\label{rem:touch boundary}
Note that we assumed that $\Lda\subset B_{1/2}$ in Theorem \ref{thm:sym} and Theorem \ref{thm:a}. Both of these two theorems also apply to  $\Lda\subset B_2$ being compact. This assumption is only used to guarantee that  $U$ is lower bounded away from zero near $\partial'' \B_2^+$. If one knows from other means that $U$ is lower bounded away from zero near $\partial'' \B_2^+$, which is indeed the case of $U$ obtained as the extension of the solution $u$ of \eqref{eq:maineq} via the Poisson integral \eqref{eq:extension integral}, then it does not matter whether $\Lda$ intersects the boundary $\partial B_2$ or not in either of these two theorems.
\end{rem}

Notice that the definition \eqref{eq:fl} makes sense when $u\in L^1(\mathbb R^n).$ we are considering compact sets with Hausdorff dimension less than $n$ (so that its Lebesgue measure is zero). Thus, singular solutions are well defined for the fractional Laplacian as long as $u\in L^1.$ In fact, one can deduce corresponding results for solutions of the nonlocal equation \eqref{eq:maineq} easily from Theorem \ref{symTemp}, \ref{thm:sym} and \ref{thm:a}. When $\sigma=1$, these are proved by Chen-Lin \cite{ChLin} by the moving plane method.  The proofs of our results are along the similar ways in Caffarelli-Jin-Sire-Xiong \cite{paper-ARMA} when $\Lda=\{0\}$, which in turn adapts ideas from Li \cite{Li06}. An important ingredient is that the equation \eqref{eq:ex0} is invariant under those Kelvin transformations with respect to the balls centered on $\pa\R^{n+1}_+$. More precisely,
for each $\bar x\in\R^n$ and $\lda>0$, we define, $\X=(\bar x,0)$, and
\be\label{kelvin}
U_{\X, \lda}(\xi):=\left(\frac{\lda}{|\xi-\X|}\right)^{n-2\sigma}U\left(\X+\frac{\lda^2(\xi-\X)}{|\xi-\X|^2}\right),
\ee
the Kelvin transformation of $U$ with respect to the ball $\B_{\lda}(\X)$. If $U$ is a solution of \eqref{eq:ex0}, then $U_{\bar X,\lda}$ is a solution of \eqref{eq:ex0} in the corresponding domain. Such conformal invariance allows us to use the moving sphere method introduced by Li-Zhu \cite{LZhu}. This observation has also been used in \cite{JLX} and \cite{paper-ARMA}. The main difficulty here is that $\Lda$ is a set of (fractional) capacity zero instead of a single point.

The organization of the paper is as follows. In the Section \ref{sec:pre}, we discuss the fractional capacity and a weighted capacity, and recall some basic properties of solutions of linear equations. From Section \ref{sec:3} to \ref{sec:5} we prove Theorem \ref{thm:a}, Theorem \ref{symTemp} and Theorem  \ref{thm:sym} in order. In the last section, we give an application of Theorem \ref{thm:a} and slightly improve a main result in Gonz\'alez-Mazzeo-Sire \cite{GMS}.

\

\noindent{\bf Acknowledgments:} the authors would like to thank the referee for his/her valuable suggestions.

\section{Preliminaries}\label{sec:pre}

Since we shall use \eqref{eq:ex0} to study \eqref{eq:maineq}, it will be convenience to give another equivalent definition of the fractional capacity $\capac_\sigma(\Lda)$ by viewing $\Lda$ as a set in $\R^{n+1}$. For a compact set $\Lda\subset\R^n$ and an open set $\Omega\subset\R^{n+1}$ satisfying $\Lda\subset\Omega$, we define
\begin{equation}\label{eq:capacity2}
\mu_\sigma(\Lda,\Omega):=\inf \left \{ \int_{\mathbb R^{n+1}} |t|^{1-2\sigma}|\nabla_X G|^2\,\ud X:\,\,G \in C^\infty_c (\Omega),\,\,G(x,0)\geq 1 \text{ in }\Lda \right \}.
\end{equation}
The functions $G \in C^\infty_c (\Omega)$ satisfying that $\,G(x,0)\geq 1 \text{ in }\Lda$ will be called admissible test functions for evaluating $\mu_\sigma(\Lda,\Omega)$. When $\Omega=\R^{n+1}$, we write $\mu_\sigma(\Lda)=\mu_\sigma(\Lda,\R^{n+1})$ for short. Notice that the weight $|t|^{1-2\sigma}$ is an $A_2$ function, and $\mu(\Lda)$ is a weighted capacity, whose general theory can be found in Fabes-Jerison-Kenig \cite{fabes} and the  book Heinonen-Kilpel\"ainen-Martio \cite{HKM}.

We are going to show that
\begin{prop}\label{prop:equivalence}
For every compact set $\Lda\subset\R^n$, there holds
\[
\mu_\sigma(\Lda)=2N(\sigma)\capac_\sigma(\Lda),
\]
where $N(\sigma)=2^{1-2\sigma}\Gamma(1-\sigma)/\Gamma(\sigma)$.
\end{prop}
\begin{proof}
For $f\in C^\infty_c (\R^n)$, let
\begin{equation}\label{eq:extension integral}
F(x,t)= \int_{\R^n} \mathcal{P}_\sigma(x-y,t) f(y)\,\ud y,
\end{equation}
where \[
 \mathcal{P}_\sigma(x,t)=\beta(n,\sigma)\frac{|t|^{2\sigma}}{(|x|^2+t^2)^{\frac{n+2\sigma}{2}}}
\] with constant $\beta(n,\sigma)$ such that $\int_{\R^n}\mathcal{P}_\sigma(x,1)\,\ud x=1$.
By   \cite{CaS}, we have
\[
\int_{\R^{n+1}_+}  t^{1-2\sigma}|\nabla_X F|^2\,\ud X =N(\sigma)  \int_{\mathbb R^n} |\xi|^{2\sigma} |\hat f(\xi)|^2\,\ud \xi.
\]

On one hand, for any $G(x,t)\in C^\infty_c(\R^{n+1}_+\cup \pa \R^{n+1}_+)$ with $G(x,0)=f(x)$, by Lemma A.4 of \cite{JLX} and the evenness of $F$ in $t$, we have
\[
\int_{\R^{n+1}_+}  |t|^{1-2\sigma}|\nabla_X F|^2 \,\ud X \le \int_{\R^{n+1}_+}  |t|^{1-2\sigma}|\nabla_X G|^2\,\ud X
\]
and
\[
\int_{\R^{n+1}_+}  t^{1-2\sigma}|\nabla_X F|^2 \,\ud X=\int_{\R^{n+1}_-}  |t|^{1-2\sigma}|\nabla_X F|^2 \,\ud X \le \int_{\R^{n+1}_-}  |t|^{1-2\sigma}|\nabla_X G|^2\,\ud X.
\]
Thus,
\[
2N(\sigma)  \int_{\mathbb R^n} |\xi|^{2\sigma} |\hat f(\xi)|^2\,\ud \xi\le \int_{\R^{n+1}}  |t|^{1-2\sigma}|\nabla_X G|^2\,\ud X,
\]
from which it follows that
\[
2N(\sigma)\capac_\sigma(\Lda)\le \mu_\sigma(\Lda).
\]

On the other hand, for every $\va>0$, there exists $f \in C^\infty_c(\mathbb R^n)$, $f(x)\geq 1$ on $\Lda$ such that
\[
\text{Cap}_\sigma(\Lda)+\va\ge \int_{\mathbb R^n} |\xi|^{2\sigma} |\hat f(\xi)|^2\,\ud \xi.
\]
Let $F$ be the one defined by $f$ through \eqref{eq:extension integral}. Let $\varphi$ be a radial smooth cut-off function supported in $\B_2$ and equal to $1$ in $\B_1$, and let $\varphi_r(X)=\varphi(X/r)$. It is elementary to check that
\[
 \int_{\mathbb R^{n+1}} |t|^{1-2\sigma}|\nabla_X (\varphi_rF)|^2\to  \int_{\mathbb R^{n+1}} |t|^{1-2\sigma}|\nabla_X F|^2\quad\mbox{as }r\to\infty.
\]
We choose $R$ large enough such that $\Lda\subset \B_{R/2}$ and
\[
 \int_{\mathbb R^{n+1}} |t|^{1-2\sigma}|\nabla_X (\varphi_RF)|^2\le\int_{\mathbb R^{n+1}} |t|^{1-2\sigma}|\nabla_X F|^2+\va
\]Let $\eta$ be a standard mollifier in $\R^{n+1}$ and $\eta_\delta(X)=\delta^{-n-1}\eta(X/\delta)$. Let $G_{\delta}=\eta_{\delta}*((1+\va)\varphi_RF)$. Then $G_\delta\in C^\infty_c (\R^{n+1})$. Since $G_\delta\to (1+\va)\varphi_RF$ uniformly in compact sets, and $(1+\va)\varphi_RF\ge 1+\va$ on $\Lda$, we have that
$G_\delta\ge 1$ on $\Lda$ for all sufficiently small $\delta$. Hence,
\[
\begin{split}
 \mu_\sigma(\Lda)&\le\int_{\mathbb R^{n+1}} |t|^{1-2\sigma}|\nabla_X G_\delta|^2\\
 &=\int_{\mathbb R^{n+1}} |t|^{1-2\sigma}|\eta_\delta*(\nabla_X ((1+\va)\varphi_RF))|^2\\
 &\to (1+\va)^2\int_{\mathbb R^{n+1}} |t|^{1-2\sigma}|\nabla_X (\varphi_RF)|^2\quad\mbox{as }\delta\to 0,
\end{split}
\]
where in the last limit we used the fact that $|t|^{1-2\sigma}$ is an $A_2$ weight.
Thus, we have
\[
\begin{split}
 \mu_\sigma(\Lda)&\le (1+\va)^2\int_{\mathbb R^{n+1}} |t|^{1-2\sigma}|\nabla_X (\varphi_RF)|^2\\
 &\le  (1+\va)^2(\int_{\mathbb R^{n+1}} |t|^{1-2\sigma}|\nabla_X F|^2+\va)\\
 &=(1+\va)^2(2N(\sigma)\int_{\mathbb R^n} |\xi|^{2\sigma} |\hat f(\xi)|^2\,\ud \xi+\va)\\
 &\le  (1+\va)^2(2N(\sigma)(\text{Cap}_\sigma(\Lda)+\va)+\va).
\end{split}
\]
Since $\va$ is arbitrary, we have
\[
 \mu_\sigma(\Lda)\le 2N(\sigma)\text{Cap}_\sigma(\Lda).
\]

This finishes the proof of this proposition.
\end{proof}

For $\Omega\subset\R^{n+1}$, let
\[
W^{1,2}(|t|^{1-2\sigma},\Omega):=\{w: w\in L^2(|t|^{1-2\sigma},\Omega), \nabla_X w\in L^2(|t|^{1-2\sigma},\Omega)\}
\]
and
\begin{equation}\label{eq:Wc}
W_c=C^0_c(\R^{n+1})\cap W^{1,2}(|t|^{1-2\sigma},\R^{n+1}).
\end{equation}
Then for every $w\in W_c$ that $w(x,0)\ge 1$ on $\Lda$, we have that
\[
\mu_\sigma(\Lda)\le \int_{\mathbb R^{n+1}} |t|^{1-2\sigma}|\nabla_X w|^2.
\]
This can be proved by the similar proof of Proposition \ref{prop:equivalence}, which is as follows. Let $\va,\delta>0$ and $\eta_\delta$ be the mollifier. Then as before, one has $(1+\va)\eta_{\delta}*w\in C^\infty_c(\R^{n+1})$ and $(1+\va)\eta_{\delta}*w\ge 1$ on $\Lda$ for all sufficiently small $\delta$. Then
\[
\begin{split}
 \mu_\sigma(\Lda)&\le\int_{\mathbb R^{n+1}} |t|^{1-2\sigma}|\nabla_X ((1+\va)\eta_{\delta}*w)|^2\\
 &= (1+\va)^2\int_{\mathbb R^{n+1}} |t|^{1-2\sigma}|\eta_\delta*(\nabla_X w)|^2\\
 &\to (1+\va)^2\int_{\mathbb R^{n+1}} |t|^{1-2\sigma}|\nabla_X w|^2\quad\mbox{as }\delta\to 0.
\end{split}
\]
By sending $\va\to 0$, we finish the proof.

This means that the admissible test functions for evaluating $\mu_\sigma(\Lda)$ in \eqref{eq:capacity2} can be chosen from a larger set $W_c$:
\[
\mu_\sigma(\Lda)=\inf \left \{ \int_{\mathbb R^{n+1}} |t|^{1-2\sigma}|\nabla_X G|^2:\,\,G \in W_c,\ G\ge 1\text{ on }\Lda \right \}.
\]
One further observes that
\[
\mu_\sigma(\Lda)=\inf \left \{ \int_{\mathbb R^{n+1}} |t|^{1-2\sigma}|\nabla_X G|^2:\,\,G \in W_c,\ 0\le G\le 1\mbox{ in }\R^{n+1}, \ G= 1\text{ on }\Lda \right \}.
\]



We have the following two properties on the connection between $\text{Cap}_\sigma(\Lda)$ and the Hausdorff measure of  $\Lda\subset\R^n$. The definition of Hausdorff measure can be found in Evans-Gariepy \cite{EG}.

\begin{thm}\label{thm:hausdroff1}
Let $\Lda\subset\R^n$ be compact. If $\MH^{n-2\sigma}(\Lda)<\infty$, then $\capac_\sigma(\Lda)=0$.
\end{thm}
\begin{proof}
We follow the proof of Theorem 3 in Section 4.7 of \cite{EG}. 

Claim: There exists a constant $C$ depending only on $n,\sigma,\Lda$ such that if $V\subset\R^{n+1}$ is any open set containing $\Lda$, there exists an open set $\Omega\subset\R^{n+1}$ and  $f\in W_c$ (defined in \eqref{eq:Wc}) such that
\[
\begin{cases}
\Lda \subset \Omega\cap\R^n, \Omega\subset \{f=1\},\\
\text{supp}(f)\subset V,\\
\displaystyle \int_{\R^{n+1}}|t|^{1-2\sigma}|\nabla_X f|^2\,\ud X\le C.
\end{cases}
\]
This claim can be proved as follows. Let $\delta=\frac 12 \text{dist}(\Lda, \R^n\setminus V)$. Since $\MH^{n-2\sigma}(\Lda)<\infty$ and $\Lda$ is compact in $\R^n$, there exists a finite collection $\{B_{r_i}(x_i)\}$ of open balls such that $2r_i<\delta, B_{r_i}(x_i)\cap \Lda\neq \emptyset, \Lda\subset\cup_{i=1}^m B_{r_i}(x_i)$ and
\[
\sum_{i=1}^mr_{i}^{n-2\sigma}\le C\MH^{n-2\sigma}(\Lda)+1
\]
for some constant $C$. Now set $\Omega=\cup_{i=1}^m \B_{r_i}(X_i)$ with $X_i=(x_i,0)$, and define $f_i$ by
\[
f_i(X)=\begin{cases}
1\quad\mbox{if }|X-(x_i,0)|\le r_i,\\
2-\frac{|X-(x_i,0)|}{r_i}\quad\mbox{if }r_i\le |X-(x_i,0)|\le 2r_i,\\
0\quad\mbox{if }|X-(x_i,0)|\ge 2r_i.
\end{cases}
\]
Then
\[
\int_{\R^{n+1}}|t|^{1-2\sigma}|\nabla_X f_i|^2\,\ud X=r_i^{n-2\sigma}\int_{\mathcal B_2(x_i,0)\setminus\mathcal B_1(x_i,0) } |t|^{1-2\sigma}\,\ud X\le Cr_i^{n-2\sigma}.
\]
Let $f=\max_{1\le i\le m} f_i$. Then $\Omega\subset \{f=1\}, \mbox{suppt}(f)\subset V$ and
\[
\int_{\R^{n+1}} |t|^{1-2\sigma}|\nabla_X f|^2\,\ud X \le \sum_{i=1}^m \int_{\R^{n+1}} |t|^{1-2\sigma}|\nabla_X f_i|^2 \,\ud X \le C \sum_{i=1}^mr_i^{n-2\sigma}\le C(\MH^{n-2\sigma}(\Lda)+1).
\]
This finishes the proof of this claim. Using the claim inductively, we can find open sets $\{V_k\}_{k=1}^\infty$ in $\R^{n+1}$ and functions $f_k\in W_c$ such that
\[
\begin{cases}
\Lda\subset V_{k+1}\cap\R^n, V_{k+1}\subset V_k,\\
\overline V_{k+1}\subset \{X\in\R^{n+1}: f_k(X)=1\},\\
\mbox{supp}(f_k)\subset V_k,\\
\displaystyle \int_{\R^{n+1}} |t|^{1-2\sigma}|\nabla_X f_k|^2\,\ud X \le C.
\end{cases}
\]
Set
\[
S_j=\sum_{k=1}^j\frac{1}{k}\quad\mbox{and}\quad g_j=\frac{1}{S_j}\sum_{k=1}^j\frac{f_k}{k}.
\]
Notice that $g_j\ge 1$ on $V_{j+1}\cap\R^{n}$, and each $g_j\in W_c$. Since $\mathrm{supp}(|\nabla_X f_k|)\subset (V_k\setminus\overline V_{k+1})$, we have
\[
\begin{split}
\text{Cap}_\sigma(\Lda)&\le \frac{1}{2N(\sigma)}\int_{\R^{n+1}} |t|^{1-2\sigma}|\nabla_X g_j|^2\,\ud X\\
&=\frac{1}{2N(\sigma)S_j^2}\sum_{k=1}^j\frac{1}{k^2}\int_{\R^{n+1}} t^{1-2\sigma}|\nabla_X f_k|^2\,\ud X\\
&\le \frac{C}{2N(\sigma)S_j^2}\sum_{k=1}^j\frac{1}{k^2}\to 0\quad\mbox{as }j\to\infty.
\end{split}
\]
\end{proof}

We will need a trace version of Poincar\'e inequalities:

\begin{lem}\label{lem:ponicare} Suppose that $f\in W^{1,2}(|t|^{1-2\sigma}, \B_{r})$.  Then
\[
\dashint_{B_r}|f-(f)_{r}|^2 \,\ud x\le C r^{2\sigma+1} \dashint_{ \B_r}|t|^{1-2\sigma}|\nabla_X f|^2\,\ud X,
\]
where $(f)_{r}=\dashint_{B_r} f(x,0)\,\ud x$, $B_r=\B_r\cap\R^n$, and $C>0$ depends only on $n,\sigma$.
\end{lem}
\begin{proof} We are going to use $f$ to denote both, the function and its restriction to $\mathbb R^n.$ By scaling, we only prove the case $r=1$. Let 
$$
g:=f-\dashint_{\B_1} f(\xi)\,\ud \xi.
$$
By the Poincar\'e inequality (Theorem 1.5 in \cite{FKS}), we have
\[
\int_{\B_1}|t|^{1-2\sigma}|g|^2\le C\int_{\B_1}|t|^{1-2\sigma}|\nabla_X g|^2,
\]where $C>0$ depends only on $n$ and $\sigma$. It follows
\[
 \|g\|^2_{W^{1,2}(|t|^{1-2\sigma},\B_1)}\le C \int_{\B_1}|t|^{1-2\sigma}|\nabla_X g|^2.
\]
 By Theorem 1.1 in \cite{Chua}, we can extend $g$ to $\tilde g\in W^{1,2}(t^{1-2\sigma},\R^{n+1})$ such that
 \[
   \|\tilde g\|^2_{W^{1,2}(|t|^{1-2\sigma},\R^{n+1})}\le C\|g\|^2_{W^{1,2}(|t|^{1-2\sigma},\B_1)},
    \]
where $C>0$ depends only on $n$ and $\sigma$.
Then we have
\[
\begin{split}
 \int_{B_1} |g(x,0)|^2\,\ud x&\le C (\int_{B_1} |g(x,0)|^{\frac{2n}{n-2\sigma}}\,\ud x)^{\frac{n-2\sigma}{n}}\\
 &\le C (\int_{\R^n} |\tilde g(x,0)|^{\frac{2n}{n-2\sigma}}\,\ud x)^{\frac{n-2\sigma}{n}} \le C    \|\tilde g\|^2_{W^{1,2}(|t|^{1-2\sigma},\R^{n+1})},
 \end{split}
\]
where we used the standard trace embedding in the last inequality. Combining all the inequalities in the above, we obtain
 \[
 \int_{B_1}|f(x,0)-\dashint_{\B_1} f(\xi)\,\ud \xi|^2 dx\le C\int_{\B_1}|t|^{1-2\sigma}|\nabla_X f|^2.
 \]
Then the conclusion follows from the fact that
\[
 \int_{B_1}|f(x,0)-\dashint_{B_1} f(y,0)\,\ud y|^2 dx\le  \int_{B_1}|f(x,0)-\dashint_{\B_1} f(\xi)\,\ud \xi|^2 dx.
\]
\end{proof}

\begin{thm}\label{thm:hausdroff2}
Assume that $\Lda\subset\R^n$ is a compact set and $\capac_{\sigma}(\Lda)=0$, then $\MH^s(\Lda)=0$ for all $s>n-2\sigma$. Namely, the Hausdorff dimension of $\Lda$ is less than or equal to $n-2\sigma$.
\end{thm}
\begin{proof}
We only need to prove it for $s$ close to $n-2\sigma$. We follow the proof of Theorem 4 in Section 4.7 of \cite{EG}. Suppose $\text{Cap}_{\sigma}(\Lda)=0$ and $s>n-2\sigma$. Then $\mu_\sigma(\Lda)=0$, and thus, for all $i\ge 1$, there exists $f_i \in C^\infty_c (\R^{n+1})$ such that $\Lda\subset \{f_i\ge 1\}$ and
\[
 \int_{\R^{n+1}} |t|^{1-2\sigma}|\nabla_X f_i|^2\,\ud X\le 2^{-i}.
\]
Let $g=\sum_{i=1}^\infty f_i$. Then
\[
( \int_{\R^{n+1}} |t|^{1-2\sigma}|\nabla_X g|^2 \,\ud X)^{1/2}\le \sum_{i=1}^\infty( \int_{\R^{n+1}} |t|^{1-2\sigma}|\nabla_X f_i|^2 \,\ud X)^{1/2}<\infty.
\]
Note that $\Lda\subset \text{interior of }\{g\ge m\}$ for all $m\ge 1$. Fix any $y\in\Lda$, then for $r$ small enough that $B_r(y)\subset \text{interior of }\{g\ge m\}$, $(g)_{y,r}:=\dashint_{B_r(y)}g\ge m$, and therefore, $(g)_{y,r}\to\infty$ as $r\to 0$. We are going to show that
\[
\limsup_{r\to 0}\frac{1}{r^{s}}\int_{\mathcal B_r(y,0)}|t|^{1-2\sigma}|\nabla_Xg|^2\,\ud X=+\infty.
\]
If not, then there exists a constant $M<\infty$ such that
\[
\frac{1}{r^s}\int_{\mathcal B_r(y,0)}|t|^{1-2\sigma}|\nabla_Xg|^2 \,\ud X\le M
\]
for all $0<r\le 1$. Then by the Poincar\'e's inequality in Lemma \ref{lem:ponicare},
\begin{equation}\label{eq:poincare}
\dashint_{B_r(y)}|g-(g)_{y,r}|^2 dx\le C r^{2\sigma+1} \dashint_{\mathcal B_r(y,0)}|t|^{1-2\sigma}|\nabla_X g|^2\,\ud X\le Mr^{s-(n-2\sigma)}.
\end{equation}
Thus,
\begin{align*}
|(g)_{y,r/2}-(g)_{y,r}|&=Cr^{-n}|\int_{B_{r/2}(y)} g-(g)_{y,r} dx|\\
&\le C\dashint_{B_r(y)}|g-(g)_{y,r}|dx\\
&\le C(\dashint_{B_r(y)}|g-(g)_{y,r}|^2dx)^{1/2}\\
&\le C r^{\frac{s-(n-2\sigma)}{2}}.
\end{align*}
Hence if $k>j$, we have
\[
|(g)_{y,1/2^k}-(g)_{y,1/2^j}|\le \sum_{i=j+1}^k|(g)_{y,1/2^i}-(g)_{y,1/2^{i-1}}|\le C\sum_{i=j+1}^k(2^{1-i})^{\frac{s-(n-2\sigma)}{2}}.
\]
Since $s>n-2\sigma$, $\{(g)_{y,1/2^k}\}_{k=1}^\infty$ is a Cauchy sequence. This contradicts with that $(g)_{y,r}\to \infty$ as $r\to 0$.

Therefore,
\begin{align*}
\Lda&\subset\{y\in\R^n: \limsup_{r\to 0}\frac{1}{r^{s}}\int_{\mathcal B_r(y,0)}|t|^{1-2\sigma}|\nabla_Xg|^2 \,\ud X=\infty\}\\
&\subset \{y\in\R^n: \limsup_{r\to 0}\frac{1}{r^{s}}\int_{\mathcal B_r(y,0)}|t|^{1-2\sigma}|\nabla_Xg|^2 \,\ud X>2\}\equiv \Lda_s.
\end{align*}
Let $V\subset\R^n$ be an open bounded neighborhood of $\Lda$ and $\delta>0$. Let $\tilde V=V\times (-R,R)$, where $R=\text{diam}(V)\le 2\text{diam}(\Lda)$. By Vitalli's covering theorem, there exists countably many disjoint open balls $\{B_{r_i}(y_i)\}_{i=1}^\infty$ such that every $B_{r_i}(y_i)\subset V$, $r_i<\delta$, $\frac{1}{r^{s}}\int_{\mathcal B_r(y_i,0)}|t|^{1-2\sigma}|\nabla_Xg|^2 \,\ud X>1$ and $\Lda\subset\cup_{i=1}^\infty B_{5r_i}(x_i)$. Thus
\[
\sum_{i=1}^\infty r_i^{s}\le C\int_{\cup_{i=1}^\infty \mathcal B_{5r_i}(x_i,0)}|t|^{1-2\sigma}|\nabla_Xg|^2 \,\ud X\le C\int_{\tilde V}|t|^{1-2\sigma}|\nabla_Xg|^2 \,\ud X.
\]
Since $s<n$, it follows that $\Lda$ is of $\R^n$-Lebesgue measure zero. Thus, we can choose $|V|_{\R^n}$ arbitrary small and thus $|\tilde V|_{\R^{n+1}}$ arbitrary small. Since
\[
\lim_{|\tilde V|\to 0}\int_{\tilde V}t^{1-2\sigma}|\nabla_Xg|^2\,\ud X=0,
\]
we have that
\[
\MH^{s}(\Lda)=0.
\]
\end{proof}

For a compact set that $\text{Cap}_\sigma(\Lda)=0$, we will have the following observation.

Suppose $\Lda\subset\R^n$ is compact such that $\text{Cap}_\sigma(\Lda)=0$. Then for every $\va$ there exists $G\in C_c^\infty(\R^{n+1})$, $G\ge 1$ on $\Lda$ such that
\[
 \int_{\mathbb R^{n+1}} |t|^{1-2\sigma}|\nabla_X G|^2\,\ud X\le \va.
\]  
We may further assume that $G\le 2$ in $\R^{n+1}$. Let $f(x)=G(x,0)$ and $F(x,t)$ be defined as in \eqref{eq:extension integral}. Then we have
\[
 \int_{\mathbb R^{n+1}} |t|^{1-2\sigma}|\nabla_X F|^2\,\ud X\le  \int_{\mathbb R^{n+1}} |t|^{1-2\sigma}|\nabla_X G|^2\,\ud X\le \va.
\]  
Let $r_0>0$ be such that $\Lda\subset B_{r_0}$, and $g\in C_c^\infty(\B_{r_0})$ that $g\ge 1$ on $\Lda$. Then $Fg\in W_c\cap C_c^0(\B_{r_0})$ and  $Fg\ge 1$ on $\Lda$. Moreover, we have
\[
\begin{split}
 \int_{\mathbb R^{n+1}} |t|^{1-2\sigma}|\nabla_X (Fg)|^2\,\ud X&\le 2 \int_{\mathbb R^{n+1}} |t|^{1-2\sigma}|\nabla_X F|^2g^2\,\ud X+ 2\int_{\mathbb R^{n+1}} |t|^{1-2\sigma}|\nabla_X g|^2F^2\,\ud X\\
 &\le 2\|g\|_{L^\infty}\va +  2\|\nabla_Xg\|_{L^\infty}\int_{\B_{r_0}} |t|^{1-2\sigma}F^2\,\ud X.
\end{split}.
\]
For the second term on the right hand side, we have
\[
\begin{split}
\int_{\B_{r_0}} |t|^{1-2\sigma}F^2\,\ud X&\le \int_{\B_{r_0}\cap\{|t|\le \va\}} |t|^{1-2\sigma}F^2\,\ud X+\int_{\B_{r_0}\cap\{|t|\ge \va\}} |t|^{1-2\sigma}F^2\,\ud X\\
&\le 4r_0^n\frac{C_n}{2-2\sigma}\va^{2-2\sigma}+(r_0^{1-2\sigma}+\va^{1-2\sigma})\int_{\B_{r_0}\cap\{|t|\ge \va\}} F^2\,\ud X.
\end{split}
\]
Furthermore,
\[
\begin{split}
\int_{\B_{r_0}\cap\{|t|\ge \va\}} F^2\,\ud X&\le C_n r_0^{1+2\sigma} (\int_{\B_{r_0}\cap\{|t|\ge \va\}} |F|^{\frac{2(n+1)}{n-2\sigma}}\,\ud X)^{\frac{n-2\sigma}{n+1}}\\
&\le C_n r_0^{1+2\sigma} (\int_{\R^{n+1}_+} |F|^{\frac{2(n+1)}{n-2\sigma}}\,\ud X)^{\frac{n-2\sigma}{n+1}}\\
&\le C_n r_0^{1+2\sigma} \|F(x,0)\|^2_{L^{\frac{2n}{n-2\sigma}(\R^n)}}\\
&\le C_n r_0^{1+2\sigma} \int_{\mathbb R^{n+1}} |t|^{1-2\sigma}|\nabla_X F|^2\,\ud X\\
&\le C_n r_0^{1+2\sigma} \va,
\end{split}
\]
where we used H\"older's inequality in the first inequality, we used Lemma 1 in \cite{Chenshibing} in the third inequality, and the trace inequality in the last inequality. Combining the above inequalities, we have
\[
\begin{split}
& \int_{\mathbb R^{n+1}} |t|^{1-2\sigma}|\nabla_X (Fg)|^2\,\ud X\\
 &\le 2\|g\|_{L^\infty}\va +  2\|\nabla_Xg\|_{L^\infty}\left(4r_0^n\frac{C_n}{2-2\sigma}\va^{2-2\sigma}+(r_0^{1-2\sigma}+\va^{1-2\sigma})C_n r_0^{1+2\sigma} \va\right).
\end{split}
\]
By taking a mollification $\eta_\delta*(2Fg)$, we obtain $h\in C^\infty_c(\B_{r_0})$ for $\delta$ small such that $h\ge 1$ on $\Lda$ and
\[
\int_{\mathbb R^{n+1}} |t|^{1-2\sigma}|\nabla_X h|^2\,\ud X\le C(n,\sigma,r_0)(\va+\va^{2-2\sigma}),
\]
where $C(n,\sigma,r_0)$ is a positive constant depending only on $n,\sigma,r_0$. Since $\va$ is arbitrary, we have that
\[
\mu_\sigma(\Lda, \B_{r_0})=0.
\]
for every $r_0>0$ such that $\Lda\subset  \B_{r_0}$.

Therefore, it follows from Lemma 2.9  in \cite{HKM} that Lemma 7.34 in \cite{HKM} applies to $\Lda$ when $\text{Cap}_\sigma(\Lda)=0$, so that $\Lda$ is removable for super-solutions. This is where we use the assumption that the singular (closed) set has zero fractional capacity. Consequently, we have the following maximum principle (Proposition \ref{prop:liminf}), which is crucial for our proofs.

We say that $U\in L^\infty_{\rm loc}(\overline {\R^{n+1}_+})$
if $U\in L^\infty(\overline {\B_R}^+ )$ for any $R>0.$ Similarly, we say $U\in W^{1,2}_{\rm loc}(t^{1-2\sigma},\B_1^+\setminus\Lda)$ if $U\in W^{1,2}(t^{1-2\sigma}, \B_2^+\setminus \overline \calO)$ for all open neighborhood $\calO\subset\overline\R^{n+1}_+$ of $\Lda$.

\begin{prop}\label{prop:liminf} Suppose $\Lda\subset \R^n$ is compact and $\capac_\sigma(\Lda)=0$,  $U\in W^{1,2}_{\rm loc}(t^{1-2\sigma},\B_1^+\setminus\Lambda)\cap C(\overline \B_1^+ \setminus \Lda)$ and
\[
\liminf_{Y\to (x,0)} U(Y)>-\infty\quad\mbox{for all }x\in\Lambda\ \ \mbox{and for all } Y\in\B_1^+.
\]
Suppose $U$ solves
\[
\begin{cases}
\mathrm{div}(t^{1-2\sigma} \nabla_{X} U)\le 0 & \quad \mbox{in }\B_1^+,\\
\frac{\pa U}{\pa \nu^\sigma} \ge 0 &\quad \mbox{on }\pa' \B_1^+\setminus \Lda,
\end{cases}
\]
in the weak sense. Then
\[
U(X)\ge\inf_{\partial'' \B_1^+} U\quad\mbox{for all }X\in \overline \B_1^+ \setminus \Lda.
\]
\end{prop}


\begin{proof} Let
$$
m:=\inf_{\partial''\mathcal{B}_{1}^{+}} U
$$
and
$$
H^{-}(x,t)=\min\{U(x,t),m\}.
$$
We make an even extension of $H^-$:
\[
H(x,t)=\begin{cases}
H^-(x,t)\quad\mbox{if }t\ge 0\\
H^-(x,-t)\quad\mbox{if }t\le 0.
\end{cases}
\]
Then it follows that
$$
\mathrm{div}\left(t^{1-2\sigma}\nabla_X H(X)\right)\leq 0 \ \ \textrm{in} \ \ \mathcal B_1\setminus\Lambda.
$$

Notice that for every $x\in\Lda\cap\B_1$ there exists $r(x)>0$ such that $H$ is bounded in $\B_{r(x)}(x)\setminus\Lda$. Since $\capac_\sigma(\Lambda)=0$, it follows from Lemma 7.34 \cite{HKM} that the set $\Lda$ is removable, that is, $H\in W^{1,2}_{loc}(|t|^{1-2\sigma},\B_1)$ and
$$
\mathrm{div}\left(t^{1-2\sigma}\nabla_X H(X)\right)\leq 0 \ \ \textrm{in} \ \ \B_1
$$
in the sense of distribution. It follows from standard maximum principle that
\begin{equation}\label{eq:infimum}
H(X)\geq \inf_{|Y|= 1}H(Y)=m, \ \ |X|\leq 1.
\end{equation}
We conclude that
$$
U(X)\geq m, \ \ X\in\overline{\mathcal B_{1}^{+}}\setminus \Lambda.
$$
\end{proof}

The following Harnack inequality will be used frequently in our proof. We state it here for convenience. See \cite{CS} or \cite{TX} for the proof.
\begin{prop}\label{prop:harnack} Let $0\leq U \in W^{1,2}(t^{1-2\sigma}, \B_R^+)$ be a weak solution of
\[
\begin{cases}
\mathrm{div}(t^{1-2\sigma} \nabla_{X} U)=0 & \quad \mbox{in }\B_R^+,\\
\frac{\pa U}{\pa \nu^\sigma} = a(x) U(x,0) &\quad \mbox{on }\pa' \B_R.
\end{cases}
\]
 If $a\in L^p(B_R)$ for some $p>n$, then we have
\[
\sup_{\overline \B_{R/2}^+} U\leq C(R) \inf_{\overline \B_{R/2}^+} U,
\]
where $C$ depends only on $n,\sigma, R$ and $\|a\|_{L^p(B_{R})}$.
\end{prop}

\section{Upper bound estimate near a singular set} \label{sec:3}

\begin{proof}[Proof of Theorem \ref{thm:a}]Suppose the contrary that there exists a sequence $\{x_j\} \subset B_1\setminus\Lda$ such that
\[
d_j:=\dist(x_j,\Lda)\to 0\quad \mbox{as } j\to \infty,
\]
but
\be\label{eq:cl1}
|d_j|^{\frac{n-2\sigma}{2}}u(x_j)\to \infty\quad \mbox{as }j\to \infty.
\ee
Without loss of generality, we may assume that $0\in\Lambda$ and $x_j\to 0$ as $j\to\infty$.

Consider
\[
v_j(x):=\left(\frac{|d_j|}{2}-|x-x_j|\right)^{\frac{n-2\sigma}{2}} u(x),\quad |x-x_j|\leq \frac{|d_j|}{2}.
\]
Let $|\bar x_j-x_j|<\frac{|d_j|}{2}$ satisfy
\[
v_j(\bar x_j)=\max_{|x-x_j|\leq \frac{|d_j|}{2}}v_j(x),
\]
and let
\[
2\mu_j:=\frac{|d_j|}{2}-|\bar x_j-x_j|.
\]
Then
\be \label{eq:cl2}
0<2\mu_j\leq \frac{|d_j|}{2}\quad\mbox{and}\quad \frac{|d_j|}{2}-|x-x_j|\ge\mu_j \quad \forall ~ |x-\bar x_j|\leq \mu_j.
\ee
By the definition of $v_j$, we have
\be \label{eq:cl3}
(2\mu_j)^{\frac{n-2\sigma}{2}}u(\bar x_j)=v_j(\bar x)\ge v_j(x)\ge (\mu_j)^{\frac{n-2\sigma}{2}}u(x)\quad \forall ~ |x-\bar x_j|\leq \mu_j.
\ee
Thus, we have
\[
2^{\frac{n-2\sigma}{2}}u(\bar x_j)\ge u(x)\quad \forall ~ |x-\bar x_j|\leq \mu_j.
\]
We also have
\be\label{eq:cl4}
(2\mu_j)^{\frac{n-2\sigma}{2}}u(\bar x_j)=v_j(\bar x_j)\ge v(x_j)= \left(\frac{|d_j|}{2}\right)^{\frac{n-2\sigma}{2}}u(x_j)\to \infty.
\ee
Now, consider
\[
W_j(y,t)=\frac{1}{u(\bar x_j)}U\left(\bar x_j+\frac{y}{u(\bar x_j)^{\frac{2}{n-2\sigma}}}, \frac{t}{u(\bar x_j)^{\frac{2}{n-2\sigma}}}\right ), \quad (y,t)\in \om_j,
\]
where
\[
\om_j:=\left\{(y,t)\in \R^{n+1}_+: \left(\bar x_j+\frac{y}{u(\bar x_j)^{\frac{2}{n-2\sigma}}},\frac{t}{u(\bar x_j)^{\frac{2}{n-2\sigma}}}\right)\in  \overline\B^+_{1}\setminus \Lda \right\}
\]
and let $w_j(x)=W_j(x,0)$ if $x\not\in\Lda$. Then $W_j$ satisfies $w_j(0)=1$ and
\be \label{eq:ext1}
\begin{cases}
\mathrm{div}(t^{1-2\sigma} \nabla W_j)=0 & \quad \mbox{in }\om_j,\\
\frac{\pa W_j}{\pa \nu^\sigma} = w_j(x)^{\frac{n+2\sigma}{n-2\sigma}} &\quad \mbox{on }\pa' \om_j.
\end{cases}
\ee
Moreover, it follows from \eqref{eq:cl3} and \eqref{eq:cl4} that
\[
w_j(y)\leq 2^{\frac{n-2\sigma}{2}} \quad\mbox{in } B_{R_j},
\]
where $R_j:=\mu_j u(\bar x_j)^{\frac{2}{n-2\sigma}}\to \infty$ as $j\to \infty$.

By Proposition \ref{prop:harnack}, for any given $\bar t>0$ we have
\[
0\leq W_j \leq C(\bar t)\quad \mbox{in }B_{R_j/2}\times [0,\bar t),
\]
where $C(\bar t)$ depends only on $n, \sigma$ and $\bar t$.
Thus, after passing to a subsequence, we have, for some nonnegative functions
$W\in W^{1,2}_{loc}(t^{1-2\sigma},\overline{\mathbb{R}^{n+1}})\cap C^{\al}_{loc}(\overline{\mathbb{R}^{n+1}})$ and $w\in C^2(\R^n)$,
\[
\begin{cases}
W_j&\rightharpoonup W\quad\mbox{weakly in }W^{1,2}_{loc}(t^{1-2\sigma},\R^{n+1}_+),\\
W_j&\rightarrow W\quad\mbox{in }C^{\al/2}_{loc}(\overline{\R^{n+1}_+}),\\
w_j&\rightarrow w\quad\mbox{in }C^2_{loc}(\R^n),
\end{cases}
\]
where $w(x)=W(x,0)$. Moreover, $W$ satisfies
\be \label{eq:ext2}
\begin{cases}
\mathrm{div}(t^{1-2\sigma} \nabla W)=0 & \quad \mbox{in }\R^{n+1}_+,\\
\frac{\pa W}{\pa \nu^\sigma} = w^{\frac{n+2\sigma}{n-2\sigma}}& \quad \mbox{on }\pa\R^{n+1}_+,
\end{cases}
\ee
and $w(0)=1$.
By the Liouville theorem in \cite{JLX}, we have,
\be \label{eq:cl5}
w(y):=W(y,0)= \left(\frac{1}{1+|y|^2}\right)^\frac{n-2\sigma}{2},
\ee
upon some multiple, scaling and translation.

On the other hand, we are going to show that
\be\label{eq:aim1}
w_{\lda, x}(y)\leq w(y)\quad \forall~\lda>0, x\in \R^n, ~ |y-x|\ge\lda.
\ee
By an elementary calculus lemma in \cite{LZhang}, \eqref{eq:aim1} implies that $w\equiv constant$. This contradicts to \eqref{eq:cl5}.

Let us fix $x_0\in\R^n$ and $\lda_0>0$. Then for all $j$ large, we have $|x_0|<\frac{R_j}{10}, 0<\lda_0<\frac{R_j}{10}$. For $\lda>0$, we let
\[
(W_j)_{X,\lda }(Y):= \left(\frac{\lda }{|Y-X|}\right)^{n-1}W_j\left(X+\frac{\lda^2(Y-X)}{|Y-X|^2}\right),
\]
for $Y\in \om_j$ with $|Y-X|\geq \lda$. Let $X_0=(x_0,0)$.

\medskip

\emph{Claim 1}: There exists a positive real number $\lda_3$ such that for any $0<\lda<\lda_3$, we have
\[
(W_j)_{X_0,\lda }(\xi)\leq W_j(\xi) \quad \text{in } \om_j\backslash \B^+_{\lda}(X_0).
\]
The proof of Claim 1 consists of two steps as the proof of Lemma 3.2 in \cite{JLX}.

\textit{Step 1.} We show that there exist $0<\lda_1<\lda_2<\lda_0$, which are independent on $j$, such that
\[
(W_j)_{X_0,\lda }(\xi)\leq W_j(\xi), ~\forall~0<\lda<\lda_1,~\lda<|\xi-X_0|<\lda_2.
\]
For every $0<\lda<\lda_1<\lda_2$, $\xi\in\pa'' \B_{\lda_2}(X_0)$, we have $X_0+\frac{\lda^2(\xi-X_0)}{|\xi-X_0|^2}\in \B^+_{\lda_2}(X_0)$. Thus we can choose $\lda_1=\lda_1(\lda_2)$ small such that
\begin{equation*}
\begin{split}
(W_j)_{X_0,\lda }(\xi)&=\left(\frac{\lambda}{|\xi-X_0|}\right)^{n-2\sigma}W_j\left(X_0+\frac{\lda^2(\xi-X_0)}{|\xi-X_0|^2}\right)\\
&\leq\left(\frac{\lda_1}{\lda_2}\right)^{n-2\sigma}\sup\limits_{\overline{\B_{\lda_2}^+(X_0)}}W_j\leq \inf_{\partial ''{\B_{\lda_2}^+(X_0)}}W_j\leq W_j(\xi),
\end{split}
\end{equation*}
where we used that $w_j\to w$ in $C^2(B_{\lda_0}(X_0))$ and Harnack inequality.
Hence
\[
(W_j)_{X_0,\lda }\leq W_j\quad \mbox{on } \partial ''(\B^+_{\lda_2}(X_0)\backslash \B^+_{\lda}(X_0))
\]
  for all $\lda_2>0$ and $0<\lda<\lda_1(\lda_2)$.

We will show that $(W_j)_{X_0,\lda }\leq W_j$ on $(\B^+_{\lda_2}(X_0)\backslash \B^+_{\lda}(X_0))$ if $\lda_2$ is small and $0<\lda<\lda_1(\lda_2)$.
Since $(W_j)_{X_0,\lda }$ also satisfies \eqref{eq:ext1} in $\B_{\lda_2}(X_0)^+\setminus\overline{\B_{\lda_1}^+(X_0)}$, we have
\begin{equation}\label{diff}
\begin{cases}
\mathrm{div}(t^{1-2\sigma}\nabla ((W_j)_{X_0,\lda }-W_j))&=0\quad \text{in}\quad \B_{\lda_2}^+(X_0)\backslash \overline{\B_{\lda}^+(X_0)},\\
\lim\limits_{t\to 0}t^{1-2\sigma}\partial_t ((W_j)_{X_0,\lda }-W_j)&\\ =
W_j^{\frac{n+2\sigma}{n-2\sigma}}(x,0)-(W_j)_{X_0,\lda }^{\frac{n+2\sigma}{n-2\sigma}}(x,0)
&\quad \text{on}\quad \partial '(\B_{\lda_2}^+(X_0)\backslash \overline{\B_{\lda}^+(X_0)}).
\end{cases}
\end{equation}
Let $((W_j)_{X_0,\lda }-W_j)^+:=\max(0,(W_j)_{X_0,\lda }-W_j)$ which equals to $0$ on $\pa''(\B^+_{\lda_2}(X_0)\backslash \B^+_{\lda}(X_0))$. Hence, by a density argument, we can use $((W_j)_{X_0,\lda }-W_j)^+$ as a test function in the definition of weak solution of \eqref{diff}. We will make use of the narrow domain technique
from \cite{BN}. With the help of the mean value theorem, we have
\begin{equation*}
\begin{split}
&\int_{\B_{\lda_2}^+(X_0)\backslash \B_{\lda}^+(X_0)} t^{1-2\sigma}|\nabla((W_j)_{X_0,\lda }-W_j)^+|^2\\
&=\int_{B_{\lda_2}(X_0)\backslash B_{\lda}(X_0)}((W_j)_{X_0,\lda }^{\frac{n+2\sigma}{n-2\sigma}}(x,0)-W_j^{\frac{n+2\sigma}{n-2\sigma}}(x,0))((W_j)_{X_0,\lda }-W_j)^+\\
&\leq C \int_{B_{\lda_2}(X_0)\backslash B_{\lda}(X_0)}(((W_j)_{X_0,\lda }-W_j)^+)^2 (W_j)_{X_0,\lda }^{\frac{4\sigma}{n-2\sigma}}\\
&\leq C\left(\int_{B_{\lda_2}(X_0)\backslash B_{\lda}(X_0)}(((W_j)_{X_0,\lda }-W_j)^+)^\frac{2n}{n-2\sigma}\right)^{\frac{n-2\sigma}{n}}\left(\int_{B_{\lda_2}(X_0)\backslash B_{\lda}(X_0)} (W_j)_{X_0,\lda }^{\frac{2n}{n-2\sigma}}\right)^{\frac{2\sigma}{n}}\\
& \leq C\left(\int_{\B_{\lda_2}^+(X_0)\backslash \B_{\lda}^+(X_0)} t^{1-2\sigma}|\nabla((W_j)_{X_0,\lda }-W_j)^+|^2\right)
\left(\int_{ B_{\lda_2}(X_0)} w_j^{\frac{2n}{n-2\sigma}}\right)^{\frac{2\sigma}{n}},
\end{split}
\end{equation*}
where Proposition 2.1 in \cite{JLX} is used in the last inequality and $C$ is a positive constant depending only on $n$ and $\sigma$. Since $w_j\to w$ in $C^2(B_{\lda_0}(X_0))$,
we can fix $\lda_2$ small independent of $j$ such that
\[
C\left(\int_{B_{\lda_2}} w_j^{\frac{2n}{n-2\sigma}}\right)^{\frac{2\sigma}{n}}<1/2.
\]
Then
\[
\nabla((W_j)_{X_0,\lda }-W_j)^+=0\quad \mbox{in } \B_{\lda_2}^+\backslash \B_{\lda}^+.
\]
Since
\[
((W_j)_{X_0,\lda }-W_j)^+=0 \quad\mbox{on }\partial ''(\B^+_{\lda_2}(X_0)\backslash \B^+_{\lda})(X_0),\]
 we have
 \[
 ((W_j)_{X_0,\lda }-W_j)^+=0\quad \mbox{in } \B_{\lda_2}^+(X_0)\backslash \B_{\lda}^+(X_0).
 \]
 We conclude that
 \[
 (W_j)_{X_0,\lda }\leq W_j \quad\mbox{in } \B^+_{\lda_2}(X_0)\backslash \B^+_{\lda}(X_0)
 \]
 for $0<\lda<\lda_1:=\lda_1(\lda_2)$.
\medskip

\textit{Step 2.}  We show that there exists $\lda_3\in (0,\lda_1)$ such that $\forall~0<\lda<\lda_3$,
\[
 (W_j)_{X_0,\lda }(\xi)\leq W_j(\xi),~\forall |\xi-X_0|>\lda_2,~\xi\in \om_j.
\]
Let $\phi(\xi)=\left(\frac{\lda_2}{|\xi-X_0|}\right)^{n-2\sigma}\inf\limits_{\partial'' \B_{\lda_2}(X_0)} W_j$,
which satisfies
\[
\begin{cases}
\begin{aligned}
  \mathrm{div}(t^{1-2\sigma}\nabla \phi)&=0&\quad& \mbox{in } \R^{n+1}_+\setminus \B_{\lda_2}^+(X_0)\\
   -\lim_{t\to 0}t^{1-2\sigma}\pa_t \phi(x,t)&=0&\quad& \mbox{on } \R^n\setminus \overline{B_{\lda_2}(X_0)},
 \end{aligned}
\end{cases}
\]
 and $\phi(\xi)\leq W_j(\xi)$ on $\partial'' \B_{\lda_2}(X_0)$. Let us examine them on $\pa''\om_j$.

 Since $u\ge 1/C>0$ on $\pa B_{3/2}$, it follows from the Harnack inequality (Proposition \ref{prop:harnack}) that
\be\label{eq:cl20}
W_j\geq \frac{1}{Cu(\bar x_j)}>0 \quad \mbox{on }\pa'' \om_j.
\ee
Note that we assumed $x_j\to 0$ without loss of generality. Then $\frac{|x_j|}{2}\le |\bar x_j|\leq \frac{3|x_j|}{2}<<1$. Thus, for any $\xi\in \pa'' \om_j$, i.e., $\left|\bar X_j+\frac{\xi}{u(\bar x_j)^{\frac{2}{n-2\sigma}}}\right|=1$, we have
\[
|\xi|\approx u(\bar x_j)^{\frac{2}{n-2\sigma}}.
\]
 Thus
 \be\label{eq:cl21}
W_j\geq \frac{1}{Cu(\bar x_j)}>\frac{1}{u^{1.5}(\bar x_j)}>\left(\frac{\lda_2}{|\xi-X_0|}\right)^{n-2\sigma}\inf\limits_{\partial'' \B_{\lda_2}(X_0)} W_j \quad \mbox{on }\pa'' \om_j,
\ee
where we used the fact that $W_j$ converges to a solution $W$ of \eqref{eq:ext2} locally uniformly in the last inequality.
By Proposition \ref{prop:liminf}, we have
\be\label{eq:cl22}
W_j(\xi)\geq \left(\frac{\lda_2}{|\xi-X_0|}\right)^{n-2\sigma}\inf\limits_{\partial'' \B_{\lda_2}(X_0)} W_j,~\forall~|\xi-X_0|>\lda_2,~\xi\in \om_i.
\ee
Let
\[
\lda_3=\min(\lda_1, \lda_2(\inf\limits_{\partial'' \B_{\lda_2}(X_0)} W_j/\sup\limits_{ \B_{\lda_2}(X_0)} W_j)^{\frac{1}{n-2\sigma}}).
\]
Then for any $0<\lda<\lda_3,~|\xi-X_0|\geq \lda_2$, $\xi\in\om_j$, we have
\[
\begin{split}
 (W_j)_{X_0,\lda }(\xi)&\leq (\frac{\lda}{|\xi-X_0|})^{n-2\sigma}W_j(X_0+\frac{\lda^2(\xi-X_0)}{|\xi-X_0|^2})\\
&\leq (\frac{\lda_3}{|\xi-X_0|})^{n-2\sigma}\sup\limits_{\B_{\lda_2}(X_0)}W_j\\
&\leq (\frac{\lda_2}{|\xi-X_0|})^{n-2\sigma}\inf\limits_{\partial ''\B_{\lda_2}(X_0)}W_j\leq W_j(\xi).
\end{split}
\]
Claim 1 is proved.

We define
\[
\bar \lda:=\sup \{0<\mu\le \lda_0 | (W_j)_{X_0,\lda}(\xi)\leq W_{j}(\xi),\quad \forall~|\xi-X_0|\geq \lda, ~\xi\in  \om_j,~\forall~ 0<\lda <\mu\}.
\]
By Claim 1, $\bar\lda$ is well defined.

\medskip

\emph{Claim 2}: $\bar\lda=\lda_0.$

\medskip

To Prove Claim 2, we argue by contradiction. Suppose $\bar\lda<\lda_0$. It follows from the strong maximum principle and \eqref{eq:cl21} that $(W_j)_{X_0,\bar\lda}(\xi)< W_{j}(\xi)$ if $|\xi-X_0|> \bar\lda, ~\xi\in \overline \om_j\setminus\Lda$. For $\delta$ small, which will be fixed later, denote $K_{\delta}=\{\xi\in\om_j: |\xi-X_0|\geq\bar\lda+\delta\}$.
Then by Proposition \ref{prop:liminf}, there exists $c_2=c_2(\delta)$ such that
\[
W_{j}(\xi)-(W_j)_{X_0,\bar\lda}(\xi)>c_2  \ \text{ in }\  K_{\delta}.
\]
By the uniform continuity of $W_j$ on compact sets, there exists $\va$ small such that for all $\bar\lda<\lda<\bar\lda+\va$
\[
(W_j)_{X_0,\bar\lda}-(W_j)_{X_0,\lda}>-c_2/2 \ \text{ in }\  K_{\delta}.
\]
Hence
\[
W_{j}-(W_j)_{X_0,\lda}>c_2/2 \ \text{ in } \ K_{\delta}.
\]
Now let us focus on the region $\{\xi\in\R^{n+1}_+: \lda\leq|\xi-X_0|\leq\bar\lda+\delta\}$.
Using the narrow domain technique as that in Claim 1, we can choose $\delta$
small (notice that we can choose $\va$ as small as we want) such that
\[
W_{j}\geq (W_j)_{X_0,\lda} \ \text{ in } \ \{\xi\in\R^{n+1}_+: \lda\leq|\xi|\leq\bar\lda+\delta\}.
\]
In conclusion, there exists $\va$ such that for all $\bar\lda<\lda<\bar\lda+\va$
\[
 (W_j)_{X_0,\lda}(\xi)\leq W_{j}(\xi),\quad \forall~|\xi-X_0|\geq \lda, ~\xi\in \overline \om_j,
\]
which contradicts with the definition of $\bar\lda$. Claim 2 is proved.

Thus
\[
 (W_j)_{X_0,\lda}(\xi)\leq W_{j}(\xi),\quad \forall~|\xi-X_0|\geq \lda,~\xi\in \overline \om_j\setminus\Lda,~\forall~0<\lda\leq \lda_0.
\]
Sending $j\to \infty$, we have
\[
w_{x_0,\lda}(y)\leq w(y)\quad \forall~0<\lda\leq \lda_0,  ~ |y-x_0|\ge\lda.
\]
Since $x_0, \lda_0$ are arbitrary, \eqref{eq:aim1} has been verified.

Theorem \ref{thm:a} is proved.
\end{proof}

\section{Symmetry for global solutions}\label{sec:4}

\begin{proof}[Proof of Theorem \ref{symTemp}]
Without loss of generality, we assume that
\begin{equation}\label{eq:one infinity}
\limsup_{x\to 0}u(x)=\infty.
\end{equation}
Denote $0_k$ as the origin in $\R^k$.

First, we would like to show that for all $y\in\R^{n-k}\setminus\{0\}$ there exists $\lda_3(y)\in (0,|y|)$ such that for all $0<\lda<\lda_3(y)$ we have
\be\label{eq:bigger}
U_{Y,\lda}(\xi)\le U(\xi)\quad\forall\ |\xi-Y|\ge \lda,~\xi\not\in\R^k\times\{0_{n-k}\}\times\{0\},
\ee
where $Y=(0_k,y,0)\in\R^{n+1}$ and
\[
U_{Y,\lda }(X):= \left(\frac{\lda }{|Y-X|}\right)^{n-2\sigma}U\left(Y+\frac{\lda^2(X-Y)}{|Y-X|^2}\right).
\]
This can be proved similarly to that for $W_j$ in the proof of Theorem \ref{thm:a}, and we sketch the proofs here.
The first step is to show that there exist $0<\lda_1<\lda_2<|y|$ such that
\[
U_{Y,\lda }(\xi)\leq U(\xi), ~\forall~0<\lda<\lda_1,~\lda<|\xi-Y|<\lda_2.
\]
The proof of this step follows exactly the same as that for $W_j$ before. The second step is to show that there exists $\lda_3(y)\in (0,|y|)$ such that \eqref{eq:bigger} holds for all $0<\lda<\lda_3(y)$. To prove this step, we only need to make sure that \eqref{eq:cl21} holds for $U$, i.e.,
\be\label{eq:cl22U}
U(\xi)\geq \left(\frac{\lda_2}{|\xi-X|}\right)^{n-2\sigma}\inf_{\partial'' \B_{\lda_2}(Y)} U,~\forall~|\xi-Y|>\lda_2,~\xi\not\in\R^k\times\{0_{n-k}\}\times\{0\},
\ee
where $\lda_2<|y|$ is small. And \eqref{eq:cl22U} can be proved as follows. Let $\Lda$ is the inversion of $\R^k$ with respect to $\partial B_{\lda_2}(y)$. So $\Lda$ is a $k$-dimensional sphere passing through $y$, and $\Lda\subset B_{\lda_2}(y)$.
\[
\begin{cases}
\mathrm{div}(t^{1-2\sigma} \nabla_{X} U_{Y,\lda_2 })=0 & \quad \mbox{in }\R^{n+1}_+,\\
\frac{\pa }{\pa \nu^\sigma} U_{Y,\lda_2}= U_{Y,\lda_2 }^{\frac{n+2\sigma}{n-2\sigma}} &\quad \mbox{on }B_{\lda_2}(y)\setminus\Lda.
\end{cases}
\]
Since $k\le n-2\sigma$, it follows from Theorem \ref{thm:hausdroff1} that $\text{Cap}_\sigma(\Lda)=0$. By Proposition \ref{prop:liminf}, we have that
\[
U_{Y,\lda_2 }(\xi)\ge \inf_{\partial'' \B_{\lda_2}(Y)} U_{Y,\lda_2 }=\inf_{\partial'' \B_{\lda_2}(Y)} U\quad\mbox{for all }\xi\in \overline\B_{\lda_2}^+(Y)\setminus \Lda.
\]
This will exactly lead to \eqref{eq:cl22U}.

Now, we can define
\[
\begin{split}
&\bar \lda(y):=\\
&\sup \{0<\mu\le |y|\ |\ U_{Y,\lda}(\xi)\leq U(\xi), \forall~|\xi-Y|\geq \lda, ~\xi\not\in \R^k\times\{0_{n-k}\}\times\{0\},~\forall~ 0<\lda <\mu\}.
\end{split}
\]

Secondly, we will show that
\be\label{eq:gseq}
\bar \lda(y)=|y|.
\ee
Suppose $\bar \lda(y)<|y|$ for some $y\neq 0$.  Notice that
\[
U_{Y,\bar\lda(y)}(\xi)\ge U(\xi)\quad\mbox{for all }\xi\in \overline \B^+_{\bar\lda(y)}(y)\setminus\Lda
\]
where $\Lda$ is the inversion of $\R^k$ with respect to $\partial B_{\bar\lda(y)}(y)$. So $\Lda\subset  B_{\bar\lda(y)}(y)$ is a $k$-dimensional sphere passing through $y$. Because of \eqref{eq:one infinity}, we know that $U_{Y,\bar\lda(y)}(\xi)\not\equiv U(\xi)$. Thus, by strong maximum principle we have $U_{Y,\bar\lda(y)}(\xi)>U(\xi)$ for $\xi\in \overline \B^+_{\bar\lda(x)}(y)\setminus\Lda$. Choose $r<\bar\lda(y)$ but close to $\bar\lda(y)$ such that $\Lda\subset  B_{r}(y)$. It follows from Proposition \ref{prop:liminf} that
\[
U_{Y,\bar\lda(y)}(\xi)-U(\xi)\ge \min_{\pa'' \B_{r}^+(Y)}(U_{Y,\bar\lda(y)}(\xi)-U(\xi))=:2c\quad\mbox{for all }\xi\in   \overline\B_{r}^+(Y)\setminus\Lda.
\]
Denote $K_{\va,\delta}=\{\xi\in\B^+_{\bar\lda(y)-\delta}(Y): \dist(\xi,\Lda)>\va\}$. We can choose $\va, \va_1$ sufficiently small ($\va_1<\va$) such that for all $\lambda\in (\bar\lda(y), \bar\lda(y)+\va_1)$,
\[
\Lda_\lda\subset\{\xi: \dist(\xi, \Lda)\le\va\}\quad \mbox{and}\quad  \{Y+\frac{\bar\lda(y)^2}{\lda^2}(\xi-Y): \dist(\xi, \Lda)\le\va\}\subset \overline\B_{r}^+(Y),
\]
where $\Lda_\lda$ is the inversion of $\R^k$ with respect to $B_\lda(y)$. Then for $\xi$ that $\dist(\xi, \Lda)\le\va$ and $\xi\not\in\Lda_\lda$,
\[
\begin{split}
U_{Y,\lda}(\xi)&=\left(\frac{\bar\lda(y)}{\lda}\right)^{n-2\sigma}U_{Y,\bar\lda(y)}\left(Y+\frac{\bar\lda(y)^2}{\lda^2}(\xi-Y)\right)\\
&\ge \left(\frac{\bar\lda(y)}{\bar\lda(y)+\va_1}\right)^{n-2\sigma}\left(U\big(Y+\frac{\bar\lda(y)^2}{\lda^2}(\xi-Y)\big)+ c\right)
\end{split}
\]
Notice that there exist $\va_1$ small that for all $\xi$ that $\dist(\xi, \Lda)\le\va$ and all $\lambda\in (\bar\lda(y), \bar\lda(y)+\va_1)$, we have
\[
\left(\frac{\bar\lda(y)}{\bar\lda(y)+\va_1}\right)^{n-2\sigma}\left(U\big(Y+\frac{\bar\lda(y)^2}{\lda^2}(\xi-Y)\big)+ c\right)\ge U(\xi)+c/2.
\]
This statement can be proved quickly by contradiction arguments. Therefore, we have shown that there exist $\va_1$ small that for all $\xi$ that $\dist(\xi, \Lda)\le\va$, $\xi\not\in\Lda_\lda$,
and all $\lambda\in (\bar\lda(y), \bar\lda(y)+\va_1)$, we have
\[
U_{Y,\lda}(\xi)\ge U(\xi)+c/2.
\]
Choose $\delta$ small, which will be fixed later, there exists $c_2>0$ such that
\[
U_{Y,\bar\lda(y)}(\xi)\ge U(\xi)+c_2\quad\mbox{for all }\xi\in K_{\delta,\va}.
\]
Since $U$ is locally uniformly continuous in $\overline \R^{n+1}_+\setminus\{\R^{k}\}$, we can choose $\va_1$ even smaller such that
\[
U_{Y,\lda}(\xi)-U_{Y,\bar\lda(y)}(\xi)\ge -c_2/2\quad\mbox{for all }\xi\in K_{\delta,\va}.
\]
Hence,
\[
U_{Y,\lda}(\xi)-U(\xi)\ge c_2/2\quad\mbox{for all }\xi\in K_{\delta,\va}.
\]
Now, in the region $\xi\in \B^+_{\lda}(Y)\setminus\B^+_{\bar\lda(y)-\delta}(Y)$, the narrow domain technique applies as before if we choose $\delta$ sufficiently small. Thus, one can get
\[
U_{Y,\lda}(\xi)\ge U(\xi)\quad\mbox{in }\B^+_{\lda}(Y)\setminus\B^+_{\bar\lda(y)-\delta}(Y).
\]
In conclusion, we have shown that there exists $\va_1>0$ such that for all $\lda\in (\bar\lda(y),\bar\lda(y)+\va_1)$,
\[
U_{Y,\lda}(\xi)\ge U(\xi)\quad\mbox{in }\B^+_{\lda}(Y)\setminus\Lda_\lda.
\]
This is a contradiction to the definition of $\bar \lda(x)$. This proved \eqref{eq:gseq}. Thus
\be\label{eq:gseq2}
\ U_{Y,\lda}(\xi)\leq U(\xi),\quad \forall~|\xi-Y|\geq \lda, ~\xi\not\in \R^k\times\{0_{n-k}\}\times\{0\},~\forall~ 0<\lda <|y|.
\ee
For any unit vector $e\in\{0_k\}\times \R^{n-k}$, for any $a>0$, $\xi=(x,z,t)\in\overline{\R^{n+1}_+}$ satisfying $(z-ae)\cdot e<0$, \eqref{eq:gseq2} holds with $y=Re$ and $\lda=R-a$. Sending $R$ to infinity, we have
\[
U(x,z,t)\ge U(x,z-2(z\cdot e-a)e,t).
\]
Since $e\in\{0_k\}\times \R^{n-k}$ and $a>0$ are arbitrary, this shows the radial symmetry in the $\R^{n-k}$-variables, and proves this theorem.
\end{proof}

\section{Asymptotic symmetry for local solutions near a singular set}\label{sec:5}
\begin{proof}[Proof of Theorem \ref{thm:sym}]
As before, we have that for all $0<\dist(x,\Lda)<\frac 14$, $X=(x,0)$,
\[
\bar \lda(x):=\sup \{0<\mu\le |x|\ |\ U_{X,\lda}(\xi)\leq U(\xi),\quad \forall~|\xi-X|\geq \lda, ~\xi\not\in\Lda,~\forall~ 0<\lda <\mu\}
\]
is well-defined and $\bar \lda(x)>0$, where we denote $\xi=(y,t)$. This statement can be proved very similarly to those in the previous two sections, as long as one notices that we can choose $\lda_2$ small such that
\be\label{eq:as}
U(\xi)\geq \left(\frac{\lda_2}{|\xi-X|}\right)^{n-2\sigma}\inf_{\partial'' \B_{\lda_2}^+(X)} U,\quad\forall~\xi\in\pa''\B^+_{1},
\ee
which implies by Proposition \ref{prop:liminf} that
\be\label{eq:as2}
U(\xi)\geq \left(\frac{\lda_2}{|\xi-X|}\right)^{n-2\sigma}\inf_{\partial'' \B_{\lda_2}^+(X)} U,\quad\forall~|\xi-X|>\lda, \xi\not\in\Lda.
\ee
For $y\in B_2$, $\frac 78\le|y|\le \frac 54$ and $0<\lda<\dist(x,\Lda)<\frac 18$,
\[
\left|x+\frac{\lda^2(y-x)}{|y-x|^2}-x\right|\le 4\lda^2\le 4\dist(x,\Lda)^2<\dist(x,\Lda)/2.
\]
Then
\[
\dist(x+\frac{\lda^2(y-x)}{|y-x|^2},\Lda)\le \left|x+\frac{\lda^2(y-x)}{|y-x|^2}-x\right|+\dist(x,\Lda)\le 3\dist(x,\Lda)/2
\]
and
\[
\dist(x+\frac{\lda^2(y-x)}{|y-x|^2},\Lda)\ge \dist(x,\Lda)-\left|x+\frac{\lda^2(y-x)}{|y-x|^2}-x\right|\ge \dist(x,\Lda)/2.
\]
It follows from Theorem \ref{thm:a} that
\[
u\left(x+\frac{\lda^2(y-x)}{|y-x|^2}\right)\le C\dist(x,\Lda)^{\frac{2\sigma-n}{2}}.
\]
Thus,
\[
u_{x,\lda}(y)=U_{X,\lda}(y,0)\le C\lda^{n-2\sigma}\dist(x,\Lda)^{\frac{2\sigma-n}{2}}\le C\dist(x,\Lda)^{\frac{n-2\sigma}{2}}
\]
for all $0<\lda<\dist(x,\Lda)<\frac 18,\ \frac 78\le|y|\le \frac 54$. By Harnack inequality in Proposition \ref{prop:harnack}, for all $|\xi|=1$, we have
\[
U_{X,\lda}(\xi)\le C\dist(x,\Lda)^{\frac{n-2\sigma}{2}}<U(\xi)\quad\forall \ 0<\lda<\dist(x,\Lda)\le \va/2,\ |\xi|=1
\]
for $\va>0$ sufficiently small. Therefore, it follows from Proposition \ref{prop:liminf} that
\[
\liminf_{\xi\to z\in\Lda} (U(\xi)-U_{X,\lda}(\xi))>c>0
\]
for some $c>0$ independent of $z\in\Lda$. As before, given these two properties with narrow domain techniques, the moving sphere procedure may continue if $\bar\lda(x)< \dist(x,\Lda)$. Thus we obtain $\bar\lda(x)=\dist(x,\Lda)$ for $0<\dist(x,\Lda)\le\va/2$, where $\va$ is sufficiently small. Thus, we have proved that there exists some constant $\va>0$ such that
\be\label{eq:small stop}
U_{X,\lda}(\xi)\le U(\xi)\quad \forall~0<\lda<\dist(x,\Lda)\le \va/2,~|\xi-X|\ge\lda,~\xi\not\in\Lda.
\ee
In particular
\be\label{eq:small stop1}
u_{x,\lda}(y)\le u(y)\quad \forall~0<\lda<\dist(x,\Lda)\le \va/2,~|y-x|\ge\lda,~y\not\in\Lda.
\ee
We can choose $\va$ even smaller so that the tubular neighborhood $N$ of $\Lda$ in Theorem \ref{thm:sym} contains the set $\{x:\dist(x,\Lda)\le\va\}$.

Let $r>0$ small (less than $\va^2$), $x_1,x_2\in \Pi_r^{-1} (z)$ be such that
\[
u(x_1)=\max_{\Pi_r^{-1} (z)} u(x),\quad u(x_2)=\min_{\Pi_{r}^{-1} (z)} u(x).
\]
Let $e_1=x_1-z, e_2=x_2-z, x_3=x_1+\va(e_1-e_2)/(4|e_1-e_2|)$. Then $e_1, e_2\in (T_z\Lda)^{\perp}$ and thus, $e_2-e_1\in (T_z\Lda)^{\perp}$. Let $\lda=\sqrt{\frac{\va}{4}(|e_1-e_2|+\frac{\va}{4})}$, which can be directly checked that $\lda<|x_3-z|=\dist(x_3,\Lda)<\va/2$. It follows from \eqref{eq:small stop1} that
\[
u_{x_3,\lda}(x_2)\le u(x_2).
\]
Notice that
\[
\begin{split}
u_{x_3,\lda}(x_2)&=\left(\frac{\lda}{|e_1-e_2|+\va/4}\right)^{n-2\sigma}u(x_1)\\
&=\left(\frac{1}{4|e_1-e_2|/\va+1}\right)^{\frac{n-2\sigma}{2}}u(x_1)\ge \left(\frac{1}{8r/\va+1}\right)^{\frac{n-2\sigma}{2}}u(x_1).
\end{split}
\]
Thus,
\[
\max_{\Pi_r^{-1} (z)} u(x)\le (8r/\va+1)^{\frac{n-2\sigma}{2}} \min_{\Pi_{r}^{-1} (z)} u(x).
\]
Thus, we have
\[
u(x)= (1+O(r))u(x')\quad\mbox{for all }x,x'\in\Pi^{-1}_r(z)\quad\mbox{as }r\to 0.
\]
Theorem \ref{thm:sym} is proved.
\end{proof}

\section{Application to the singular fractional Yamabe problem on conformally flat manifolds} \label{sec:6}

In this section we give an application to the singular Yamabe problem, slightly improving a theorem in \cite{GMS}. Problem \eqref{eq:maineq} arises in the study of the fractional version of the singular Yamabe problem, which has been initiated in \cite{GMS,CG}.  The original Yamabe problem is related to the so-called conformal laplacian on a compact manifold $(M,g)$, i.e.,
$$L_g=\frac{4(n-1)}{n-2} \Delta_g+R_g,$$
where $R_g$ is the scalar curvature of $M$. Generalizations of this operator (in the covariant framework) are known as GJMS operators \cite{GJMS}. The conformal laplacian is conformally covariant in the following sense: if $f$ is any (smooth) function and $\bar g = u^{\frac{4}{n-2}}\,  g$ for some $u > 0$, then
\begin{equation}
L_g(uf) = u^{\frac{n+2}{n-2}}L_{\bar g} (f).
\label{eq:cccL}
\end{equation}
 Higher order versions of this operator are denoted $P^{\overline g}_k$,
which exist for all $k \in {\mathbb N}$ if $n$ is odd, but only for $k \in \{1, \ldots, n/2\}$ if $k$ is even. The first construction
of these operators, by Graham-Jenne-Mason-Sparling \cite{GJMS}. This leads naturally to the question whether there exist any conformally covariant pseu\-do\-dif\-fe\-ren\-tial operators of noninteger
order. A partial result in this direction was given by Peterson \cite{P}, who showed that for any $\sigma$, the
conformal covariance condition determines the full Riemannian symbol of a pseudodifferential operator with
principal symbol $|\xi|^{2\sigma}$.
The breakthrough result, by Graham and Zworski \cite{GZ}, was that if $(M,[\bar g])$ is a smooth compact manifold
endowed with a conformal structure, then the operators $P^{\overline g}_k$ can be realized as residues at the values $\sigma= k$ of
the meromorphic family $S(n/2 + \sigma)$ of scattering operators associated to the Laplacian on any Poincar\'e-Einstein manifold
$(X,G)$ for which $(M,[\bar g])$ is the conformal infinity.   These are the `trivial' poles of the scattering operator, so-called
because their location is independent of the interior geometry; $S(s)$ typically has infinitely many other poles, which are
called resonances.
Multiplying this scattering family by some $\sigma$ factors to regularize these poles, one obtains a holomorphic family of
elliptic pseudodifferential operators $P_\sigma^{\bar g}$. An alternate construction of
these operators has been obtained by Juhl, and his monograph \cite{Juhl} describes an intriguing general framework for
studying conformally covariant operators. The operators $P^{\bar g}_\sigma$ are elliptic of order $2\sigma$ with principal symbol $|\xi|^{2\sigma}_{\bar g}$; finally, we have the following covariance property
\begin{equation}
\mbox{if}\ g=u^{\frac{4}{n-2\sigma}}\bar g, \qquad \mbox{then}\ P_\sigma^{\bar g} (uf) = u^{\frac{n+2\sigma}{n-2\sigma}} P_\sigma^g (f)
\label{eq:ccfcL}
\end{equation}
for any smooth function $f$. Generalizing the formul\ae\ for scalar curvature and  the Paneitz-Branson $Q$-curvature (when
$\sigma = 2$), we make the definition that, for any $0 < \sigma< n/2$, the quantity $Q_\sigma^{\bar g},$ which we call the $Q$-curvature of order $\sigma$ associated to a metric $\bar g,$ is given by
\begin{equation}
Q_\sigma^{\bar g} = P_\sigma^{\bar g}(1).
\label{eq:Qgamma}
\end{equation}

It is interesting to construct complete metrics of constant (positive) $Q_\sigma$ curvature on open
subdomains $\Omega = M \setminus \Lambda$, or in other words, to find metrics $g = u^{4/(n-2\sigma)}\bar g$ which are complete
on $\Omega$ and such that $u$ satisfies the Yamabe equation for the operator $P^\sigma_g$ with $Q_\sigma$ a constant.  This is the fractional singular Yamabe
problem. As a matter of fact  if $u$ is a solution of \eqref{eq:maineq} then the metric
$$
g=u^{\frac{4}{n-2\sigma}}|dx|^2
$$
has constant ($\equiv 1$) fractional curvature $Q_\sigma$ and is singular along $\Lambda$. The following has been proved in \cite{GMS}.
\begin{thm}\label{th:SY}
Suppose that $(M^n,\bar g)$ is compact and $g = u^{\frac{4}{n-2\sigma}}\bar g$ is a complete conformally flat metric on $\Omega = M \setminus
\Lambda$,  where $\Lambda$ is a smooth $k$-dimensional submanifold with $k \leq n-2\sigma$.  Assume furthermore that $u$ is polyhomogeneous
along $\Lambda$. If $0 < \sigma <n/2 $, and if $Q_\sigma^g > 0$ everywhere
for any choice of asymptotically Poincar\'e-Einstein extension $(X,G)$ which defines $P_\sigma^{\bar g}$ and hence
$Q_\sigma^g$, then $n$, $k$ and $\sigma$ are restricted by the inequality
\begin{equation}
\Gamma(\frac{n}{4} - \frac{k}{2} + \frac{\sigma}{2}) \Big/ \Gamma(\frac{n}{4} - \frac{k}{2} - \frac{\sigma}{2}) > 0,
\label{eq:dimrest}
\end{equation}
where $\Gamma$ is the ordinary Gamma function.  This inequality holds in particular when $k < (n-2\sigma)/2$,
and in this case then there is a unique extension of $u$ to a distribution on all of $M$ which solves the same equation,
or in other words, $u$ extends uniquely to a weak solution on all of $M$.
\label{th:fSY}
\end{thm}
Recall that $u$ is said to be polyhomogeneous along $\Lambda$ if in terms of any cylindrical coordinate system $(r,\theta,y)$
in a tubular neighbourhood of $\Lambda$, where $r$ and $\theta$ are polar coordinates in disks in the normal bundle and
$y$ is a local coordinate along $\Lambda$, $u$ admits an asymptotic expansion
\[
u \sim \sum a_{jk}(y,\theta) r^{\mu_j} (\log r)^k
\]
where $\mu_j$ is a sequence of complex numbers with real part tending to infinity, for each $j$, $a_{jk}$ is nonzero
for only finitely many nonnegative integers $k$, and such that every coefficient $a_{jk} \in \calC^\infty$. The number $\mu_0$
is called the leading exponent $\Re (\mu_j) > \Re (\mu_0)$ for all $j \neq 0$.

Thanks to Theorem \ref{thm:a} it is possible to weaken the degree of polyhomogeneity of the leading exponent of $u$ in the previous theorem.  But at the same time, we have to assume that the fractional capacity of the singular set is zero.

\begin{cor}
Suppose that $(M^n,\bar g)$ is compact and $g = u^{\frac{4}{n-2\sigma}}\bar g$ is a complete conformally flat metric on $\Omega = M \setminus
\Lambda$,  where $\Lambda$ is a smooth $k$-dimensional submanifold with $k\leq n-2\sigma$.  Assume furthermore that $u$ is polyhomogeneous
along $\Lambda$ with leading exponent $\alpha$ such that
$$
\alpha \geq (n-2\sigma)/2.
$$ If $0 < \sigma <1 $, and if $Q_\sigma^g > 0$ everywhere
for any choice of asymptotically Poincar\'e-Einstein extension $(X,G)$ which defines $P_\sigma^{\bar g}$ and hence
$Q_\sigma^g$, then $n$, $k$ and $\sigma$ are restricted by the inequality
\begin{equation}
\Gamma(\frac{n}{4} - \frac{k}{2} + \frac{\sigma}{2}) \Big/ \Gamma (\frac{n}{4} - \frac{k}{2} - \frac{\sigma}{2}) > 0,
\end{equation}
where $\Gamma$ is the ordinary Gamma function.
\end{cor}

\begin{proof}Since the new metric is conformally flat, the conformal factor $u$ satisfies the critical equation \eqref{eq:maineq}. Let $u$ be a polyhomogeneous distribution on $M$ with singular
set along the smooth submanifold $\Lambda$. Because of the bound in Theorem \ref{thm:a} and Remark \ref{rem:touch boundary}, the leading term in the expansion of $u$ is $a(y) r^{-n/2 + \sigma}$. One is then in the framework of \cite{GMS} and the proof follows.
\end{proof}

\bigskip

\noindent T. Jin

\noindent Department of Mathematics, The Hong Kong University of Science and Technology\\
Clear Water Bay, Kowloon, Hong Kong

\smallskip
and
\smallskip

\noindent Department of Computing and Mathematical Sciences, California Institute of Technology \\
1200 E. California Blvd., MS 305-16, Pasadena, CA 91125, USA\\[1mm]
Email: \textsf{tianlingjin@ust.hk} / \textsf{tianling@caltech.edu}

\medskip

\noindent O. S. de Queiroz

\noindent Universidade Estadual de Campinas, IMECC,
Departamento de Matem\'atica\\
Rua S\'ergio Buarque de Holanda, 651, Campinas, SP, Brazil. CEP 13083-859\\
Email: \textsf{olivaine@ime.unicamp.br}

\medskip

\noindent Y. Sire

\noindent Johns Hopkins University, Department of mathematics,\\
Krieger Hall, Baltimore, USA\\
Email: \textsf{sire@math.jhu.edu}
\medskip

\noindent J. Xiong

\noindent  School of Mathematical Sciences, Beijing Normal University\\
Beijing 100875, China\\[1mm]
Email: \textsf{jx@bnu.edu.cn}

\end{document}